\documentclass[11pt]{amsart}

\usepackage{amsmath}
\usepackage{amssymb}
\usepackage{graphicx}

\newtheorem{theorem}{Theorem}[section]
\newtheorem{corollary}[theorem]{Corollary}
\newtheorem{lemma}[theorem]{Lemma}

\theoremstyle{definition}
\newtheorem{definition}[theorem]{Definition}

\newtheorem{remark}[theorem]{Remark}
\newcommand{\Eps}{{\mathcal{\varepsilon}}}
\newcommand{\sign}{\mathop{\rm sgn}\nolimits}
\numberwithin{equation}{section}

\newcommand{\ds}{\displaystyle}

\title[piecewise $q$-monotone trigonometric approximation ]{No Jackson-type estimates for piecewise $q$-monotone, $q\ge3$, trigonometric approximation}

\author[ D. Leviatan] {Dany Leviatan}
\address[D. Leviatan]{Raymond and Beverly Sackler School of Mathematical
Sciences, Tel Aviv University, Tel Aviv 69978, Israel}
\email{\tt leviatan@tauex.tau.ac.il}

\author[O. V. Motorna]{Oksana  V. Motorna}
\address[O. V. Motorna]{Faculty of Radio Physics, Electronics and Computer Systems, Taras Shevchenko National University of Kyiv, 01601 Kyiv, Ukraine}
\email{\tt omotorna@ukr.net}

\author[ I. A. Shevchuk]{Igor A. Shevchuk}
\address[I. A. Shevchuk]{Faculty of Mechanics and Mathematics, Taras
Shevchenko National University of Kyiv, 01601 Kyiv, Ukraine} \email{\tt
shevchuk@univ.kiev.ua}



\begin{document}

\begin{abstract}
We say that a function $f\in C[a,b]$ is $q$-monotone, $q\ge3$, if $f\in C^{q-2}(a,b)$ and $f^{(q-2)}$ is convex in
$(a,b)$. Let $f$ be continuous and $2\pi$-periodic, and change its $q$-monotonicity finitely many times in $[-\pi,\pi]$. We are interested in estimating the degree of approximation of $f$ by trigonometric polynomials which are co-$q$-monotone with it, namely, trigonometric polynomials that change their $q$-monotonicity exactly at the points where $f$ does. Such Jackson type estimates are valid for piecewise monotone ($q=1$) and piecewise convex (q=2) approximations. However, we prove, that no such estimates are valid, in general, for co-$q$-monotone approximation, when $q\ge3$.
\end{abstract}

\maketitle

\centerline{Dedicated to the memory of our friend Professor Gerhard Opfer}
\centerline{15.9.1935--20.2.2020}

\section{Introduction and the main results}
A function $f\in C[a,b]$ is called $q$-monotone, $q\ge2$, $q\in \mathbb{N}$, if $f\in C^{q-2}(a,b)$ and $f^{(q-2)}$ is a convex function on
$(a,b)$. For the sake of uniformity, for $q=1$, we say that $f\in C[a,b]$ is $1$-monotone, if it is nondecreasing in $[a,b]$.

Let $s\in\mathbb N$ and $\mathbb Y_s:=\{Y_s\}$ where $Y_s=\{y_i\}_{i=1}^{2s}$ such that $y_{2s}<\cdots<y_1<y_{2s}+2\pi=:y_0$. We say that a $2\pi$-periodic function $f\in C(\mathbb{R})$, is piecewise $q$-monotone with respect to $Y_s$, if it changes its $q$-monotonicity at the points $Y_s$, that is, if $(-1)^{i-1}f$ is  $q$-monotone on $[y_i,y_{i-1}]$, $1\le i\le 2s$.  We denote by $\Delta^{(q)}(Y_s)$ the collection of all such piecewise $q$-monotone functions. Note that if, in addition, $f\in C^q(\mathbb{R})$, then  $f\in\Delta^{(q)}(Y_s)$, if and only if,
$$
f^{(q)}(t)\prod_{i=1}^{2s}(t-y_i)\ge0,\quad t\in[y_{2s},y_0].
$$
\begin{remark} We do not consider the case of odd number of points $y_i\in[-\pi,\pi)$, since, in that case, the collection  $\Delta^{(q)}(Y_s)$ consists only of constant functions. Note that we also have excluded the case $s=0$, namely, the collection of $q$-monotone functions, as, again, it consists only of constant functions.
\end{remark}
We also need the notation $W^r,\quad r\in\mathbb{N},$ for the Sobolev class of $2\pi$-periodic functions $f\in AC^{(r-1)}(\mathbb{R})$, such that
$$
\|f^{(r)}\|\le2.
$$
For a $2\pi$-periodic function $g$, denote
$$
\|g\|:={\rm esssup}_{x\in\mathbb{R}}|g(x)|.
$$
If, in addition, $g$ is continuous, then, of course,
$$
\|g\|=\max_{x\in\mathbb{R}}|g(x)|.
$$
Similarly, for a function $g$, defined on the interval $[a,b]$, we denote  $\|g\|_{[a,b]}:={\rm esssup}_{x\in[a,b]}|g(x)|$, and if $g\in C[a,b]$, then $\|g\|_{[a,b]}=\max_{x\in[a,b]}|g(x)|$.

Let ${\mathcal T}_n$ be the space of trigonometric  polynomials
$$
T_n(t)=\alpha_0+\sum_{k=1}^n(\alpha_k\cos kt+\beta_k\sin kt),\quad \alpha_k\in \mathbb{R},\ \beta_k\in \mathbb{R},
$$
of degree $\le n$ (of order $2n+1$) and, for $2\pi$-periodic function $g\in C(\mathbb{R})$, let
$$
E_n(g):=\inf_{T_n\in{\mathcal T}_n}\|g-T_n\|,
$$
denote the error of the best approximation of the function $g$. If $g\in\Delta^{(q)}(Y_s)$, then we would like to approximate it by trigonometric polynomials that change their $q$-monotonicity together with $g$, namely, are in $\Delta^{(q)}(Y_s)$. We call it co-$q$-monotone approximation. Denote by
$$
E_n^{(q)}(g,Y_s):=\inf_{T_n\in{\mathcal T}_n\cap\Delta^{(q)}(Y_s)}\|g-T_n\|,
$$
the error of the best co-$q$-monotone approximation of the function $g$.

It is well known that if $f\in\Delta^{(1)}(Y_s)\cap W^r$, $r\ge1$, then for $q=1$,
\begin{equation}\label{1}
E_n^{(q)}(f,Y_s)=O(1/n^r),\quad n\to\infty,
\end{equation}
(see, e.g., \cite{D,{LZ},P} for details and references).

It is also known \cite{Z} that \eqref{1}, with $q=2$, is valid for $f\in\Delta^{(2)}(Y_s)\cap W^r$, when $r\le3$. We believe it is true for all $r\ge1$.

It turns out, and proving this is the main purpose of this article, that for $q\ge 3$, \eqref{1} is, in general, invalid for any $r,s\in\mathbb N$ and every $Y_s\in\mathbb Y_s$.

Our main result is
\begin{theorem}\label{1.1} For each $q\ge 3$, $r\in\mathbb{N}$, $s\in\mathbb{N}$ and any $Y_s\in\mathbb{Y}_s$, there exists a function $f\in\Delta^{(q)}(Y_s)\cap W^r$, such that
$$
\limsup_{n\to\infty}n^rE_n^{(q)}(f,Y_s)=\infty.
$$
\end{theorem}
We will also prove the following less general but more precise statements.
\begin{theorem}\label{1.2}
For each $q\ge3$, $s\in\mathbb{N}$ and any $Y_s\in\mathbb{Y}_s$, there exists a function $f\in\Delta^{(q)}(Y_s)\cap W^{q-2}$, such that
\begin{equation}\label{2}
E_n^{(q)}(f,Y_s)\ge C(q,Y_s), \quad n\in\mathbb{N},
\end{equation}
where $C(q,Y_s)>0$ depends only on $q$ and $Y_s$.
\end{theorem}
\begin{corollary}\label{1.0}
For each  $q\ge3$, $r\le q-2, s\in\mathbb{N}$ and any $Y_s\in\mathbb{Y}_s$, there exists a function $f\in\Delta^{(q)}(Y_s)\cap W^r$, such that
\[
E_n^{(q)}(f,Y_s)\ge C(q,Y_s), \quad n\in\mathbb{N},
\]
where $C(q,Y_s)>0$ depends only on $q$ and $Y_s$.
\end{corollary}

\begin{theorem}\label{1.3}
For each $q\ge3$, $s\in\mathbb{N}$ and any $Y_s\in\mathbb{Y}_s$, there exists a function $f\in\Delta^{(q)}(Y_s)\cap W^{q-1}$, such that
\begin{equation}\label{3}
nE_n^{(q)}(f,Y_s)\ge C(q,Y_s), \quad n\in\mathbb{N},
\end{equation}
where $C(q,Y_s)>0$ depends only on $q$ and $Y_s$.
\end{theorem}
Our final result is

\begin{theorem}\label{z} Let $q\ge3$, $p\ge q$, $s\in\mathbb{N}$ and $Y_s\in\mathbb{Y}_s$. For each sequence $\{\varepsilon_n\}_{n=1}^\infty$ of positive numbers, tending to infinity, there is a function
$f\in\Delta^{(q)}(Y_s)\cap W^p$, such that
$$
\limsup_{n\to\infty}\varepsilon_n n^{p-q+2}E_n^{(q)}(f,Y_s)=\infty.
$$
\end{theorem}

We prove Theorem \ref{1.2} in Section 2, Theorem \ref{1.3} in Section 4 and Theorem \ref{z} in Section 6. In the proofs we apply ideas from \cite{LS}, and we have to overcome the constraints and challenges of periodicity.

In the sequel, positive constants $c$ and $c_i$ either are absolute or may depend only  on  $r$, $q$, $p$ and $m$.

\section{Eulerian type ideal splines and proof of Theorem  \ref{1.2}}

\begin{definition} For each $b\in(0,\pi]$ and $r\in \mathbb{N}$ denote by $\Eps_{r,b}$ the $2\pi$-periodic function,
 such that
\begin{itemize}
\item[1)] $\Eps_{r,b}\in C^{r-1}$,\\
\item[2)] $\ds\int_{-\pi}^\pi\Eps_{r,b}(x)dx=0,\quad\text{and}$\\
\item[3)] $\Eps_{r,b}^{(r)}=\sign x-\gamma_b,\quad x\in(-b,2\pi-b)\setminus\{0\},$
\end{itemize}
where
\begin{equation}\label{gamma}
\gamma_b=1-b/\pi,
\end{equation}
so that
$$
\int_{-\pi}^\pi\Eps_{r,b}^{(r)}(x)dx=0.
$$
\end{definition}
\begin{remark} By its definition, $\Eps_{r,b}$ is a spline of minimal defect of degree $r$, in particular
$\Eps_{r,\pi}$ is called an Eulerian ideal spline.
\end{remark}
Put
\[
F_r(x):=\frac{1}{r!}|x|x^{r-1}.
\]
The following properties of $\Eps_{r,b}$ readily follow from its definition.
\begin{equation}\label{p}
\Eps_{r,b}(x)=F_r(x)+p_{r,b}(x),\quad x\in[-b,2\pi-b],
\end{equation}
where $p_{r,b}$ is an algebraic polynomial of degree $\le r$;
\begin{equation}\label{77}
1\le\|\Eps_{r,b}^{(r)}\|<2,\quad\text{whence}\quad \Eps_{r,b}\in W^r,
\end{equation}
and, for each collection $Y_s$, such that $\{-b,0\}\in Y_s$, and every $q>r$, we have
\begin{equation}\label{79}
\Eps_{r,b}\in\Delta^{(q)}(Y_s).
\end{equation}
We need the following lemma (see \cite[Lemma 2.4]{LS}).
\begin{lemma}\label{2.2} For each $q\ge3$ and any function $g\in
C^{q-2}[-1,1]$, such that $g^{(q-2)}$ is convex on $[0,1]$ and concave on $[-1,0]$, 
 we have
\begin{equation}\label{ls2}
\|F_{q-2}-g\|_{[-1,1]}\ge c.
\end{equation}
\end{lemma}

\begin{proof}[Proof of Theorem \ref{1.2}] Given $Y_s\in\mathbb Y_s$, let
$$
b:= \min_{1\le j\le2s}\{y_{j-1}-y_j\},
$$
and by shifting the periodic function $f$, we may assume, without loss of generality, that $y_{2s}=-b$ and $y_{2s-1}=0$. Obviously, it follows that $y_{2s-2}\ge b$.

We will show that $f:=\Eps_{q-2,b}$ is the desired function. Indeed, by (\ref{77}) and (\ref{79}),
 $\Eps_{q-2,b}\in\Delta^{(q)}(Y_s)\cap W^{q-2}$. So we have to prove \eqref{1.2}.

To this end we take an arbitrary polynomial $T_n\in\mathcal T_n\cap\Delta^{(q)}(Y_s)$.
Then the function $g_n:=T_n-p_{q-2,b}$ satisfies $xg_n^{(q)}(x)\ge0$ for $x\in[-b,b]$, whence $xg_n^{(q)}(x/b)\ge0$ for $x\in[-1,1]$. Let $\tilde F_{q-2}(x):=F_{q-2}(x/b)$, $\tilde g_{n}(x):=g_{n}(x/b)$. By Lemma \ref{2.2} we obtain,
\begin{align*}
\|f-T_n\|_{[-\pi,\pi]}&=\|F_{q-2}-g_n\|_{[-\pi,\pi]}\ge\|F_{q-2}-g_n\|_{[-b,b]}\\&=\|\tilde F_{q-2}-\tilde g_n\|_{[-1,1]}
=b^{2-q}\| F_{q-2}-b^{q-2}\tilde g_n\|_{[-1,1]}\\&\ge{b^{2-q}}c,
\end{align*}
which yields \eqref{1.2}.
\end{proof}

\section{Approximation of $|x|$}
Recall that
$$
F_1(x)\equiv|x|.
$$
In this Section we prove, for trigonometric polynomials, an analog of Bernstein's estimate
$$
\|F_1-P_n\|_{[-b,b]}\ge c\frac bn,
$$
which is valid for every algebraic polynomial $P_n$ of degree $\le n$ (for the exact constant $c$, see \cite{5}).

To this end, we first extend to an arbitrary interval $[-b,b]$ the
Bernstein -- de la Vall\'ee-Poussin inequality
\begin{equation}\label{BV}
\|T_n'\|\le n\|T_n\|,
\end{equation}
which is valid for every $T_n\in\mathcal T_n$.

Following Privalov (see \cite[p. 96-97]{Pr}), we prove
\begin{lemma}\label{deriv1} For each  $b\in(0,\pi)$ and every odd trigonometric polynomial $T_n\in\mathcal T_n$ there holds the inequality
\begin{equation} \label{dz1}
\|T_n'\|_{[-b/2,b/2]}\le  \frac{c_0}bn\|T_n\|_{[-b,b]},
\end{equation}
where $c_0<10$.
\end{lemma}
\begin{proof}
Let
$$
\varphi(x):=\sqrt{1-x^2}.
$$
Then by virtue of \eqref{BV}, for every algebraic polynomial, $P_{n-1}$, of degree $<n$, we have
\begin{equation}\label{deriv21}
|\varphi(x)(P_{n-1}(x)\varphi(x))'|\le n\|P_{n-1}\varphi\|_{[-1,1]},\quad x\in(-1,1),
\end{equation}
Indeed, given $P_{n-1}$, set $x:=\cos t$, and denote by
$$
\tau_n(t):=P_{n-1}(\cos t)\sin t,
$$
the odd trigonometric polynomial of degree $\le n$. Then
$$
\bigl|\frac d{dt}\tau_n(t)\bigr|=\bigl|\varphi(x)\frac d{dx}(P_{n-1}(x)\varphi(x))\bigr|.
$$
Thus, \eqref{BV} readily implies
\begin{align*}
\bigl|\varphi(x)\frac d{dx}(P_{n-1}(x)\varphi(x))\bigr|=\bigl|\frac d{dt}\tau_n(t)\bigr|\le n \|\tau_n\|_{[-\pi,\pi]}=n\|P_{n-1}\varphi\|_{[-1,1]}.
\end{align*}

For each $-1<h<1$, let
$$
\varphi_h(x):=\sqrt{(1-x)(x-h)}.
$$
Then the linear mapping of $[-1,1]$ onto $[h,1]$, turns \eqref{deriv21} into,
\begin{equation}\label{deriv3}
|\varphi_h(x)(P_{n-1}(x)\varphi_h(x))'|\le n\|P_{n-1}\varphi_h\|_{[h,1]},\quad x\in(h,1).
\end{equation}
Let $h:=\cos b$. Since, for $x\in[h,1)$,
$$
\frac{\varphi_h^2(x)}{\varphi^2(x)}=\frac{x-h}{1+x}\le\frac{1-h}2=\sin^2(b/2)\le\frac{b^2}4,
$$
that is
\begin{equation}\label{deriv4}
\left\|\varphi_h/\varphi\right\|_{[h,1]}\le b/2,
\end{equation}
and
$$
-\frac{\varphi_h(x)}{\varphi(x)}\frac d{dx}\left(\frac{\varphi(x)}{\varphi_h(x)}\right)=\frac12\frac{1+h}{1+x}\frac1{x-h} \le\frac12\frac1{x-h},
$$
so that \eqref{deriv3} implies, for all $x\in(h,1)$,
\begin{align}\label{88}
&|(P_{n-1}(x)\varphi(x))'|\\&\,\le|(P_{n-1}(x)\varphi_h(x))'|\frac{\varphi(x)}{\varphi_h(x)}
+\left|P_{n-1}(x)\varphi_h(x)\left(\frac{\varphi(x)}{\varphi_h(x)}\right)'\right|\nonumber\\
&\,\le \|P_{n-1}\varphi\|_{[h,1]}\left(\frac{nb}2\frac{\varphi(x)}{\varphi_h^2(x)}+\frac{\varphi_h(x)}{\varphi(x)}\left|\left(\frac{\varphi(x)}{\varphi_h(x)}\right)'\right|\right)\nonumber\\
&\,\le \frac12\|P_{n-1}\varphi\|_{[h,1]}\left({nb}\frac{\varphi(x)}{\varphi_h^2(x)}+\frac1{x-h}\right).\nonumber
\end{align}
Finally, if  $x\in[\cos(b/2),1)$, then
\begin{equation}\label{deriv6}
\frac{\varphi(x)}{x-h}\le\frac{\sin(b/2)}{\cos(b/2)-\cos b}=\frac{\cos(b/4)}{\sin(3b/4)}\le \cot(b/4)<\frac{ 4}{b},
\end{equation}
and
\begin{align}\label{deriv5}
\frac{\varphi^2(x)}{\varphi_h^2(x)}&=\frac{1+x}{x-h}\le\frac{1+\cos(b/2)}{\cos(b/2)-\cos b}\\&=\frac{\cos^2(b/4)}{\sin(b/4)\sin(3b/4)}\le \cot^2(b/4)
<\frac{16}{b^2}.\nonumber
\end{align}

We are ready to complete the proof. Given odd trigonometric polynomial $T_n\in\mathcal T_n$, denote by $P_{n-1}$ the algebraic polynomial of degree $<n$, such that
$$
T_n(t)\equiv P_{n-1}(\cos t)\sin t.
$$
For $t\in[-b/2,b/2]$ (thus, $x=\cos t\in[\cos(b/2),1]$), combining \eqref{88} through \eqref{deriv5}, we get
\begin{align*}
\bigl|\frac d{dt}T_n(t)\bigr|&=\bigl|\varphi(x)\frac d{dx}\bigl(P_{n-1}(x)\varphi(x)\bigr)\bigr|
\le\frac12\|P_{n-1}\varphi\|_{[h,1]}\left({nb}\frac{\varphi^2(x) }{\varphi_h^2(x)} +\frac{\varphi(x) }{x-h}\right)\\
&<\frac2b(4n+1)\|P_{n-1}\varphi\|_{[h,1]}=\frac2b(4n+1)\|T_{n}\|_{[-b,b]}.
\end{align*}
This completes the proof.
\end{proof}
\begin{remark} Similarly (more easily) one may prove \eqref{dz1} for even polynomials $T_n\in\mathcal T_n$ (in fact, with 10 replaced by 4) and, in turn, for any $T_n\in\mathcal T_n$ (see, e.g., Lemma \ref{pri} below).
\end{remark}
We are ready to prove Lemma \ref{mod}.  We follow the arguments in [DzSh, pages 434,435].
\begin{lemma}\label{mod} For each $b\in(0,\pi]$ and   polynomial $T_n\in\mathcal T_n$  we have
\begin{equation}\label{modul}
\|F_1-T_n\|_{[-b,b]}\ge\frac{c_1b}n,
\end{equation}
where $c_1\ge(80c_0)^{-1}$.
\end{lemma}
\begin{proof}
We may assume that $c_0>1$. Let
\[
c^*:=\frac1{40c_0}<\frac1{40}.
\]
Assume to the contrary, that there is a polynomial $\tilde T_n\in\mathcal T_n$, such that
\begin{equation}\label{modul8}
\|F_1-\tilde T_n\|_{[-b,b]}<\frac{c^*b}{2n}.
\end{equation}
Then there is an even polynomial $\hat T_n\in\mathcal T_n$, such that
\begin{equation}\label{modul2}
\|F_1-\hat T_n\|_{[-b,b]}\le\frac{c^*b}{n}
\end{equation}
and
\begin{equation}\label{modul3}
\hat T_n(0)=0.
\end{equation}
Hence $\hat T_n$ may be represented in the form
$$
\hat T_n(t)=a_1(1-\cos t)+\dots+a_n(1-\cos nt)=2\sum_{k=1}^na_k\sin^2\left(\frac{kt}2\right).
$$
Thus, for $T_n(t):=\hat T_n(2t)$ we have
\begin{equation}\label{modul7}
\|2F_1-T_n\|_{[-b/2,b/2]}\le\frac{c^*b}{n}
\end{equation}
Denote
$$
\tau_n(t):=\frac{T_n(t)}{\sin t},\quad (\tau_n(0)=T'_n(0)).
$$
Then $\tau_n$ is an odd trigonometric polynomial of degree $<2n$.

First, we prove that
\begin{equation}\label{modul4}
\|\tau_n\|_{[-b/2,b/2]}< 10.
\end{equation}
Indeed, by virtue of (\ref{modul7}), one has
$$
\|T_n\|_{[-b/2,b/2]} \le b+\frac{c^*b}{n}<\frac{9b}8.
$$
Hence, if $b/8\le|t|\le b/2$, then
$$
|\tau_n(t)|<\frac{9b}{8\sin(b/8)}\le\frac{9\pi}{8\sin(\pi/8)}<\frac{9\pi}{6\sin(\pi/6)}=3\pi<10.
$$
Thus, assuming the contrary, that there is a point $t_0\in[-b/2,b/2]$, such that
$$
\|\tau_n\|_{[-b/2,b/2]}=|\tau_n(t_0)|=M\ge10,
$$
we conclude, that $t_0\in[-b/8,b/8]$. Since Lemma \ref{deriv1} implies
$$
\|\tau'_n\|_{[-b/4,b/4]}\le\frac{2c_{0}}b(2n-1)M,
$$
we get for
$$
t\in I_n:=\left[t_0-\frac{c^*b}n,t_0+\frac{c^*b}n\right]\subset(-b/4,b/4),
$$
\begin{align*}
|\tau_n(t)|&\ge |\tau_n(t_0)| -|\tau_n(t)-\tau_n(t_0)|\ge|\tau_n(t_0)|-|t-t_0|\|\tau'_n\|_{I_n} \\
&\ge M - |t-t_0|\|\tau'_n\|_{[-b/4,b/4]}\ge M - |t-t_0|\frac{2c_0}b(2n-1)M\\
&\ge M-2c_0c^*\frac{2n-1}nM>M-M/3=\frac23 M.
\end{align*}
Hence, for $t\in I_n$,
$$
|T_n(t)|\ge\frac23 M |\sin t|\ge\frac{M}3|t|,
$$
which, in turn, implies
\begin{align*}
\|T_n-2F_1\|_{I_n}&\ge\left(\frac{M}3-2\right)\|F_1\|_{I_n}\ge\left(\frac{10}3-2\right)\frac {c^*b}n\\
&>\frac{c^*b}n,
\end{align*}
contradicting \eqref{modul7}. Therefore, \eqref{modul4} is proved.

By virtue of Lemma \ref{deriv1} and \eqref{modul4},
$$
\|\tau'_n\|_{[-b/4,b/4]}\le\frac{2c_0}b(2n-1)\|\tau_n\|_{[-b/2,b/2]}<\frac{40c_0}bn=\frac n{c^*b}.
$$
Therefore,  for $t\in(0,b/4]$,
$$
\left|\tau_n(t)\right|=\left|\int_0^tQ'_n(u)\,du\right|< \frac{tn}{c^*b},
$$
whence
$$
|T_n(t)|<\frac{tn}{c^*b}\sin t<\frac{t^2n}{c^*b}.
$$
Hence, for
$$
t=\frac{c^*b}n
$$
we get
$$
2t-T_n(t)> t\left(2-\frac{tn}{c^*b}\right)=t=\frac{c^*b}n,
$$
contradicting \eqref{modul7} and, in turn, \eqref{modul8}. This completes the proof.
\end{proof}
We need the following simple lemma.
\begin{lemma}\label{pri} If $f\in C[-a,a]$ is an even function and $g\in C[-a,a]$ is an odd function, then
$$
\|f\|_{[-a,a]}\le\|f+g\|_{[-a,a]}\quad\text{and}\quad \|g\|_{[-a,a]}\le\|f+g\|_{[-a,a]}.
$$
\end{lemma}
\begin{proof} Let $M:=\|f+g\|_{[-a,a]}$ and assume to the contrary, that there is a point $x\in[-a,a]$, such that $|f(x)|=K>M$.
Then either $|f(x)+g(x)|\ge K$, or $|f(-x)+g(-x)|=|f(x)-g(x)|\ge K$, a contradiction. The proof for $g$ is similar.
\end{proof}
\begin{corollary} \label{modd} For each $b\in(0,\pi]$, any linear function $l$ and every trigonometric  polynomial $T_n\in\mathcal T_n$  we have
\begin{equation}\label{modull}
\|F_1+l-T_n\|_{[-b,b]}\ge\frac{c_1b}n.
\end{equation}
\end{corollary}
\begin{proof} We represent $T_n$ in the form $T_n=T_e+T_o$, where $T_e$ is an even polynomial, and $T_o$ is
an odd polynomial. Let $l(x)=ax+k=:l_o(x)+l_e$. Denote $\tilde T_e:=T_e-l_e\in\mathcal T_n$, the even polynomial. By \eqref{modul}, $\|F_1-\tilde T_e\|\ge{c_1b}/n$.
Since $l_o-T_o$ is an odd function, it follows by Lemma \ref{pri} that \eqref{modull} is valid.
\end{proof}

\section{Proof of Theorem \ref{1.3}}
The following result readily follows from \cite[Lemma 3.1]{LS}.
\begin{lemma}\label{3.111} Given $q\ge3$.
\noindent
If a function $f\in C^{q-2}[-2b,2b]$ has a convex $(q-2)$-nd derivative  $f^{(q-2)}$ on $[0,2b]$ and a concave $(q-2)$-nd derivative  $f^{(q-2)}$ on $[-2b,0]$, then
\begin{equation}\label{311}
b^{q-2}\|f^{(q-2)}\|_{[-b,b]}\le c_2\|f\|_{[-2b,2b]}.
\end{equation}
\end{lemma}
Indeed, let $\|f^{(q-2)}\|_{[-b,b]}\ne0$ and $x^*\in[-b,b]$ be such that
$
|f^{(q-2)}(x^*)|=\|f^{(q-2)}\|_{[-b,b]}.
$
If either $x^*=0$ and $f^{(q-2)}(0)<0$, or $x^*>0$, then \cite[(3.1)]{LS}  yields,
$$
b^{q-2}\|f^{(q-2)}\|_{[-b,b]}=b^{q-2}\|f^{(q-2)}\|_{[0,b]}\le c_2\|f\|_{[0,2b]}\le c_2\|f\|_{[-2b,2b]}.
$$
Otherwise (\ref{311}) follows from \cite[(3.2)]{LS}.

Recall that $F_{r}(x)={|x|x^{r-1}}/r!$. We have,
\begin{lemma}\label{3.2} For every $b\in(0,\pi]$, every trigonometric  polynomial
$T_n\in\mathcal T_n$, satisfying $tT_n^{(r+1)}(t)\ge0$ for $|t|\le b$, and any
algebraic polynomial $P_r$ of degree $\le r$, we have
\begin{equation}\label{4}
n\|F_r+P_r-T_n\|_{[-b,b]}\ge c_3b^r,\quad n\in\mathbb{N}.
\end{equation}
\end{lemma}

\begin{proof}
Since $F_{r}^{(r-1)}=F_1$ and $P_{r}^{(r-1)}$ is linear,
it follows by Corollary \ref{modd} that
$$
\|T_n^{(r-1)}-F_{r}^{(r-1)}-P_{r}^{(r-1)}\|_{[-b/2,b/2]}\ge\frac{c_1b}{2n}.
$$
Now, $T_n^{(r-1)}-F_{r}^{(r-1)}-P_{r}^{(r-1)}$ is convex in $[0,b]$ and concave in $[-b,0]$, so by virtue of Lemma \ref{3.111},
$$
\|T_n-F_{r}-P_{r}\|_{[-b,b]}\ge\frac1{c_2}\left(\frac b2\right)^{r-1}\|T_n^{(r-1)}-F_{r}^{(r-1)}-P_{r}^{(r-1)}\|_{[-b/2,b/2]}\ge\frac{c_1}{c_2n}\left(\frac b2\right)^{r}.
$$
Hence, \eqref{4} follows with $c_3\ge 2^{-r}c_1/c_2,$.
\end{proof}

\begin{proof}[Proof of Theorem \ref{1.3}]
Given $Y_s\in\mathbb Y_s$, again, let
\begin{equation}\label{b1}
b:= \min_{j\in\mathbb{Z}}\{y_{i+1}-y_i\},
\end{equation}
and by shifting the periodic function $f$, we may assume, without loss of generality, that $y_{2s}=-b$ and $y_{2s-1}=0$.
Then $f:=\Eps_{q-1,b}$ is the desired function. Indeed, by \eqref{77} and \eqref{79},
$\Eps_{q-1,b}\in\Delta^{(q)}(Y_s)\cap W^{q-1}$. So we have to prove \eqref{3}.

To this end, take an arbitrary polynomial $T_n\in\mathcal T_n\cap\Delta^{(q)}(Y_s)$. By \eqref{p},
$$
\Eps_{q-1,b}(x)=F_{q-1}(x)+p_{q-1,b}(x),\quad x\in[-b,2\pi-b],
$$
where $p_{q-1,b}$ is an algebraic polynomial of degree $\le q-1$. Therefore, Lemma \ref{3.2} implies \eqref{3} with $C(q,Y_s)\ge c_3b^{q-1}$.
\end{proof}

\section{Auxiliary results}
Let $S\in C^\infty(\mathbb{R})$, be a monotone odd function, such that $S(x)=\sign x$, $|x|\ge1.$

Put
$$
s_j:=\|S^{(j)}\|,\quad j\in\mathbb N_0.
$$
Fix $d\in(0,\pi]$, and for $\lambda\in(0,d/3]$, let
$$
\tilde S_{\lambda,d}(x):=
\begin{cases}
S\left(\frac {x-2\lambda}\lambda \right),&\text{if}\quad x\in[0,2\pi-d],\\
-S\left(\frac {x-2\lambda+d}\lambda \right),&\text{if}\quad x\in[-d,0].
\end{cases}
$$
Finally, denote
$$
S_{\lambda,d}(x):=\tilde S_{\lambda,d}(x)-\gamma_d,\quad x\in[-d,2\pi-d],
$$
where $\gamma_d$ was defined in \eqref{gamma}.

Note that
\begin{equation}\label{6}
\|S_{\lambda,d}^{(j)}\|=\lambda^{-j}s_j,\qquad j\in\mathbb N.
\end{equation}
and
$$
\int_{-\pi}^\pi S_{\lambda,d}(x)dx=0.
$$
\begin{definition} For each $\lambda\in(0,d/3]$ and $r\in \mathbb{N}$ denote by $\Eps_{r,d,\lambda}$ the $2\pi$-periodic function
$\Eps_{r,d,\lambda}\in C^\infty(\mathbb{R})$, such that
\begin{itemize}
\item[1)] $\ds\int_{-\pi}^\pi\Eps_{r,d,\lambda}(x)dx=0,\quad\text{and}$\\
\item[2)] $\Eps_{r,d,\lambda}^{(r)}=S_{\lambda,d}(x), \quad x\in[-d,2\pi-d]$.
\end{itemize}
\end{definition}
Note that for each $j\in\mathbb N$, we have
\begin{equation}\label{606}
[-d,2\pi-d]\cap\text{supp}\,\Eps_{r,d,\lambda}^{(r+j)}=[-d+\lambda,-d+3\lambda]\cup[\lambda,3\lambda],
\end{equation}
and that \eqref{6} implies
\begin{equation}\label{66}
\|\Eps_{r,d,\lambda}^{(r+j)}\|=\lambda^{-j}s_j,\qquad j\in\mathbb{N}.
\end{equation}
Also,
\begin{equation}\label{667}
\|\Eps_{r,d,\lambda}^{(j)}\|<c_4,\quad j=0,\dots,r,\quad\text{in particular}\quad \|\Eps_{r,d,\lambda}^{(r)}\|<2.
\end{equation}
\begin{lemma} We have
\begin{equation}\label{666}
\|\Eps_{r,d,\lambda}-\Eps_{r,d}\|\le c_5\lambda.
\end{equation}
\end{lemma}
\begin{proof} Put $\Eps_j:=\Eps_{j,d}-\Eps_{j,d,\lambda}$, $j=1,\dots,r$. Since $\int_{-\pi}^\pi\Eps_j(x)dx=0$, it follows that for any $1\le j\le r$ there is an $x_j\in[-\pi,\pi]$ such that $\Eps_j(x_j)=0$. Hence, we first conclude that
$$
\|\Eps_1\|\le\int_{-d}^{2\pi-d}|\sign x-\tilde S_{\lambda,d}(x)|\,dx=8\lambda.
$$
Assume by induction that $\|\Eps_{j}\|\le c\lambda$ for some $j<r$, and note that $\Eps'_{j+1}=\Eps_j$. Thus, for $x\in[x_{j+1}-\pi,x_{j+1}+\pi]$,
$$
|\Eps_{j+1}(x)|=|\Eps_{j+1}(x)-\Eps_{j+1}(x_{j+1})|=|\int_{x_{j+1}}^x\Eps_j(t)\,dt|\le \pi c\lambda,
$$
and the proof is complete.
\end{proof}
\begin{lemma}\label{5.3} Let $0< b\le d$ and $r\in \mathbb N$ be given. Let $n\in\mathbb{N}$, and let $T_n\in\mathcal T_n$, be such that $tT_n^{(r+1)}(t)\ge0$ for $|t|\le b$.
Then for any algebraic polynomial $P_r$ of degree $\le r$, if
$$
0<\lambda\le\min\left\{\frac{c_3b^r}{2nc_5},\frac d3\right\}=:\min\left\{c_6\frac{b^r}n,\frac d3\right\},
$$
then
\begin{equation}\label{44}
2n\|\Eps_{r,d,\lambda}+P_r-T_n\|_{[-b,b]}\ge c_3b^r.
\end{equation}
\end{lemma}
\begin{proof} Inequalities \eqref{4} and \eqref{666} imply
\begin{align*}
2n\|\Eps_{r,d,\lambda}+P_r-T_n\|_{[-b,b]}&\ge 2n\|\Eps_{r,d}+P_r-T_n\|_{[-b,b]}-2n\|\Eps_{r,d,\lambda}-\Eps_{r,d}\|\\
&\ge2c_3b^r-2nc_5\lambda\ge c_3b^r.
\end{align*}
This completes the proof.
\end{proof}
Fix $r\ge2$ and $m\in \mathbb{N}$, and let $q:=r+1$, and
$$
c_7:=c_6^m s_m^{-1}.
$$
For $0<b\le d$ and each $n\ge3c_6b^r$, denote,
$$
\lambda_{n,b}:=c_6\frac{b^r}n,
$$
and
$$
f_{n,b}:=c_7\frac{b^{rm}}{n^m}\Eps_{r,d,\lambda_{n,b}}.
$$
Then, we have
\begin{lemma}\label{aux}
We have,
\begin{equation}\label{662}
\|f_{n,b}^{(r+m)}\|\le1,
\end{equation}
\begin{equation}\label{663}
\|f_{n,b}^{(r+j)}\|\le c_8n^{j-m},\quad j=0,\dots,m,
\end{equation}
and
\begin{equation}\label{664}
\|f_{n,b}^{(j)}\|\le\frac{ c_9}{n^m},\quad j=0,\dots,r.
\end{equation}
For each collection $Y_s$, such that $y_{2s}=-d$, $y_{2s-1}=0$, and $d=\min_{1\le j\le2s}\{y_{j-1}-y_j\}$, we have
\begin{equation}\label{661}
f_{n,b}\in\Delta^{(q)}(Y_s),
\end{equation}
and for every polynomial $T_n\in\mathcal T_n$, satisfying $tT_n^{(q)}(t)\ge0$ for $|t|\le b$  and any algebraic polynomial $P_r$ of degree $\le r$,  we have
\begin{equation}\label{441}
n^{m+1}\|f_{n,b}+P_r-T_n\|_{[-b,b]}\ge c_{10}b^{r(m+1)}.
\end{equation}
\end{lemma}
\begin{proof}
First, \eqref{664} and \eqref{661} are clear from the definition of $\Eps_{r,d,\lambda_{n,b}}$ and \eqref{667}, respectively.

We prove \eqref{662} and \eqref{663} together. By virtue of \eqref{66}, we have, for $j=0,\dots,m$,
\begin{align*}
\|f_{n,b}^{(r+j)}\|
=&c_7\frac{b^{rm}}{n^m}\left(c_6\frac{b^r}n\right)^{-j}s_j=c_6^m s_m^{-1}\frac{b^{rm}}{n^m}\left(c_6\frac{b^r}n\right)^{-j}s_j\\
=&c_6^{m-j}n^{j-m}b^{r(m-j)}\frac{s_j}{s_m},
\end{align*}
that is, \eqref{662} and \eqref{663}.

Finally, we prove (\ref{441}). Let $\tilde P_r:=\left(c_7\frac{b^{rm}}{n^m}\right)^{-1}P_r$, $\tilde T_r:=\left(c_7\frac{b^{rm}}{n^m}\right)^{-1}T_r$,
apply Lemma \ref{5.3}  and get
\begin{align*}
n^{m+1}\|f_{n,b}+P_r-T_n\|_{[-b,b]}&=n^{m+1}c_7\frac{b^{rm}}{n^m}\|\Eps_{r,b,\lambda}+\tilde P_r-\tilde T_n\|_{[-b,b]}\\
&\ge n^{m+1}c_7\frac{b^{rm}}{n^m}\frac{c_3b^r}{2n}\\
&=:c_{10}b^{r(m+1)}.
\end{align*}
\end{proof}

\section{Proof of Theorem \ref{z}}
Set $r:=q-1$ and $m:=p-r$. Given $Y_s\in\mathbb Y_s$, let
$$
d:= \min_{1\le j\le2s}\{y_{j-1}-y_j\},
$$
and by shifting the periodic function $f$, we may assume, without loss of generality, that $y_{2s}=-d$ and $y_{2s-1}=0$. Obviously, it follows that $y_{2s-2}\ge d$.

We will prove, that the desired function $f$ may be taken in the form
$$
f(x):=\sum_{k=1}^\infty f_{n_{k+1},b_k},
$$
where integers $n_k$ and numbers $b_k$ are chosen as follows. We put $n_1:=\lceil 3c_6d^r\rceil$ and $b_1:=d/4$. Then, let $n_2$ be such that $b_2:=\lambda_{n_2,b_1}<b_1/3$.
Assume that $n_k$ and $b_k$ have been chosen. Then we take $n_{k+1}\ge2n_k$, to be such that
\begin{equation}\label{91}
3\lambda_{n_{k+1},b_k}<b_k,
\end{equation}
\begin{equation}\label{92}
\varepsilon_{n_{k+1}}c_{10}b_k^{r(m+1)}\ge k,
\end{equation}
and
\begin{equation}\label{93}
\frac{c_9}{n_{k+1}^m}\le\frac{c_{10}b_{k-1}^{r(m+1)}}{10n_k^{m+1}}.
\end{equation}
Denote
\begin{equation}\label{95}
b_{k+1}:=\lambda_{n_{k+1},b_k}.
\end{equation}
It follows by \eqref{606} and  \eqref{95} that for any $j\in\mathbb N$,
\begin{equation}\label{463}
[-d,2\pi-d]\cap\text{supp}\,f_{n_{k+1},b_k}^{(r+j)}=[-d+b_{k+1},-d+3b_{k+1}]\cup[b_{k+1},3b_{k+1}].
\end{equation}
Hence by \eqref{91}, for any $j\in\mathbb N$,
\begin{equation}\label{464}
\text{supp}\,f_{n_{k+1},b_k}^{(r+j)}\cap\text{supp}\,f_{n_{k},b_{k-1}}^{(r+j)}=\emptyset.
\end{equation}

We divide the proof of Theorem 1.6 into two Lemmas.
\begin{lemma}\label{96} We have
\begin{equation}\label{6.1}
f\in W^p\cap\Delta^{(q)}(Y_s).
\end{equation}
\end{lemma}
\begin{proof} Inequalities (\ref{663}) and (\ref{664}) imply, for all $j=0,\dots p-1$,
$$
\|f_{n_{k+1},b_k}^{(j)}\|\le \frac c{n_{k+1}},\quad k\in\mathbb{N}.
$$
Hence, for each $j=0,\dots p-1$,
$$
\sum_{k=1}^\infty \|f_{n_{k+1},b_k}^{(j)}\|\le c\sum_{k=1}^\infty\frac1{n_{k+1}}\le\frac c{n_2}\sum_{k=1}^\infty\frac1{2^j}=c,
$$
so that $f$ is well defined on $\mathbb R$, it is periodic, $f\in C^{p-1}$, for each $j=0,\dots, p-1$,
$$
f^{(j)}(x)\equiv\sum_{k=1}^\infty f_{n_{k+1},b_k}^{(j)}(x),
$$
which, combined with \eqref{661}, implies that $f\in\Delta^{(q)}(Y_s)$.

Then (\ref{464}) means, that for each point $x\in(-d,0)\cup(0,2\pi-d)$ there is neighbourhood, where the sum in $f^{(r+j)}$ consists of at most one term not identically zero. Hence,
$f\in C^\infty((-d,0)\cup(0,2\pi-d))$ and, in particular,
$
f\in C^p((-d,0)\cup(0,2\pi-d)).
$
Combining with (\ref{662}), we have $\|f^{(p)}\|\le1$, and the proof is complete.
\end{proof}
\begin{lemma}\label{98} For each $k>2$, we have
\begin{equation}\label{fin}
n_k^{m+1}\varepsilon_{n_k}E_{n_k}^{(q)}(f,Y_s)\ge k/2.
\end{equation}
\end{lemma}
\begin{proof}
Fix $k>1$. Then by (\ref{91}) and (\ref{95}), for every $1\le j\le k-1$,
$$
f_{n_{j+1},b_j}^{(r+1)}(x)=0,\quad \text{if}\quad |x|\le b_k.
$$
Hence,
\begin{equation}\label{317}
P_r(x):=\sum_{j=1}^{k-1}f_{n_{j+1},b_j}(x),\quad |x|\le b_k,
\end{equation}
is an algebraic polynomial of degree $\le r$.

Now, by (\ref{664}) and (\ref{93}),
\begin{align}\label{318}
\sum_{j=k+1}^{\infty}\|f_{n_{j+1},b_j}\|\le& c_9\sum_{j=k+1}^{\infty}\frac1{n_{j+1}^m}\le\frac{c_9}{n_{k+2}^m}\sum_{j=0}^\infty\frac1{2^{jm}}=\frac{2c_9}{n_{k+2}^m}\\
\le&\frac{c_{10}b_k^{r(m+1)}}{5n_{k+1}^{m+1}}\nonumber
\end{align}
Finally, we take an arbitrary polynomial $T_{n_{k+1}}\in\mathcal T_{n_{k+1}}\cap\Delta^{(q)}(Y_s)$ and note, that $tT_{n_{k+1}}^{(q)}\ge 0$ for $|t|\le b_k\le d$.
Therefore (\ref{317}), (\ref{318}) and (\ref{441}), imply
\begin{align*}
\|f-T_{n_{k+1}}\|\ge&\|f-T_{n_{k+1}}\|_{[-b_k,b_k]}=\bigl\|P_r+\sum_{j=k}^{\infty}f_{n_{j+1},b_j}-T_{n_{k+1}}\bigr\|_{[-b_k,b_k]}\\
=&\bigl\|\bigl(P_r+f_{n_{k+1},b_k}-T_{n_{k+1}}\bigr)+\sum_{j=k+1}^{\infty}f_{n_{j+1},b_j}\bigr\|_{[-b_k,b_k]}\\
\ge&\bigl\|P_r+f_{n_{k+1},b_k}-T_{n_{k+1}}\bigr\|_{[-b_k,b_k]}-\bigl\|\sum_{j=k+1}^{\infty}f_{n_{j+1},b_j}\bigr\|\\
\ge& \frac{c_{10}b_k^{r(m+1)}}{n_{k+1}^{m+1}}-\frac{c_{10}b_k^{r(m+1)}}{5n_{k+1}^{m+1}}=\frac{4c_{10}b_k^{r(m+1)}}{5n_{k+1}^{m+1}}.
\end{align*}
Combining with \eqref{92}, we obtain \eqref{fin}, and the proof is complete.
\end{proof}

\end{document}